\documentclass{article}

\usepackage{tkz-graph}
\usepackage{amsthm}
\usepackage{graphicx}
\usepackage{framed}
\theoremstyle{definition}
\newtheorem{theorem}{Theorem}[section]
\newtheorem{lemma}[theorem]{Lemma}
\newtheorem{proposition}[theorem]{Proposition}
\newtheorem{definition}[theorem]{Definition}
\newtheorem{remark}[theorem]{Remark}

\newtheorem{corollary}[theorem]{Corollary}

\newcommand{\adn}{\ensuremath{\gamma_{\mathrm{aut}}}}
\newcommand{\edn}{\ensuremath{\gamma_\infty}}
\newcommand{\Acal}{\ensuremath{\mathcal{A}}}
\newcommand{\Bcal}{\ensuremath{\mathcal{B}}}
\newcommand{\Ccal}{\ensuremath{\mathcal{C}}}
\newcommand{\Dcal}{\ensuremath{\mathcal{D}}}
\newcommand{\Ecal}{\ensuremath{\mathcal{E}}}
\newcommand{\Fcal}{\ensuremath{\mathcal{F}}}

\title{Autonomous Domination}
\author{Peter Ulrickson}
\date{}

\begin{document}

\SetGraphUnit{2}

\maketitle
\abstract{The well-known notion of domination in a graph abstracts the idea of protecting locations with guards. This paper introduces a new graph invariant, the autonomous domination number, which abstracts the idea of defending a collection of locations with autonomous agents following a simple protocol to coordinate their defense using only local information.}

\tableofcontents

\section{Introduction}

\subsection{Domination and Variations}

A dominating set in a graph is a subset of vertices satisfying the condition that every vertex is either contained in the subset or is adjacent to such a vertex. One thinks of the members of the set as location at which guards are placed. A guard protects its own location as well as all locations that are immediately adjacent. The domination number of a graph is the least size of a dominating set.

One can ask how well a configuration of guards (i.e. a dominating set) can respond to an attack, or series of attacks, by an adversary. In the first case, that of a single attack, one arrives at the notion of secure dominating set. A dominating set is said to be securely dominating if, given an attack, a guard can move to the attack location and the new guard configuration is still dominating. No condition is made, though, on what might happen in a subsequent attack.

One can consider an arbitrarily long sequence of attacks. This leads to eternal dominating sets and foolproof eternal dominating sets. The former case refers to a configuration of guards which can respond to any series of attacks of arbitrary length under the direction of a strategic planner. The latter refers to a guard configuration which responds to attacks but without any coordination whatsoever, either by a master planner or among the guards. Instead, any guard adjacent to an attack moves to defend it.

There is a natural intermediate notion between eternal domination and foolproof eternal domination, called autonomous domination, and it is the purpose of the present work to introduce and study this idea. Autonomous domination amounts to abstracting the idea of general purpose guards which follow simple rules, communicating with their neighbors. In response to an attack, a guard will move to address the attack, provided that other guards can still cover the locations that become exposed by the guard's movement. In the event that multiple guards can move to the attack, a randomization protocol picks one arbitrarily.

\subsection{Two Basic Examples}
\label{sec:basicExamples}
Consider a simple path on four vertices, defended by two guards, as in the following figure. The guards form a dominating set.
\vspace{.1in}

\begin{center}
\begin{tikzpicture}[scale=.5]
  \Vertex[x=0,y=0,NoLabel]{a}
  \Vertex[x=2,y=0,NoLabel]{b}
  \Vertex[x=4,y=0,NoLabel]{c}
  \Vertex[x=6,y=0,NoLabel]{d}
  \Edges(a,b)
  \Edges(b,c)
  \Edges(c,d)
  \AddVertexColor{gray}{b}
  \AddVertexColor{gray}{d}
\end{tikzpicture}
\end{center}
\vspace{.1in}

Suppose an attack is made at an intermediate vertex, indicated in red below.
\vspace{.1in}

\begin{center}
\begin{tikzpicture}[scale=.5]
  \Vertex[x=0,y=0,NoLabel]{a}
  \Vertex[x=2,y=0,NoLabel]{b}
  \Vertex[x=4,y=0,NoLabel]{c}
  \Vertex[x=6,y=0,NoLabel]{d}
  \Edges(a,b)
  \Edges(b,c)
  \Edges(c,d)
  \AddVertexColor{gray}{b}
  \AddVertexColor{gray}{d}
  \AddVertexColor{red}{c}
\end{tikzpicture}
\end{center}
\vspace{.1in}

Both guards are adjacent to the problem vertex. Only one guard should respond to the attack, though, namely the one to the right. If the guard to the left moves, the defensive position would become unbalanced, exposing the leftmost vertex to subsequent attack. The guards can coordinate their response using only locally available information. The right guard would not lose coverage of any vertex by moving, while the left guard would. Therefore there is a preferred defense that the guards can determine on their own, without higher order coordination.

The preceding example highlights the difference between autonomous domination and what is called foolproof eternal domination. The latter involves guards which respond to attacks without any strategy or coordination at all, so that two guards do not suffice to guard the path, in the foolproof sense.

A second example illustrates the significance of the presence of a master planner. Suppose a graph is defended by three guards, marked in gray in the following figure.

\begin{center}
\begin{tikzpicture}[scale=.5]
  \Vertex[x=0,y=0,NoLabel]{a}
  \Vertex[x=8,y=0,NoLabel]{b}
  \Vertex[x=4,y=-1,NoLabel]{c}
  \Vertex[x=4,y=-3,NoLabel]{p}
  \Vertex[x=4,y=-5,NoLabel]{e}
  \Vertex[x=4,y=-7,NoLabel]{f}
  \AddVertexColor{gray}{a}
  \AddVertexColor{gray}{b}
  \AddVertexColor{gray}{c}
  \Edges(a,b)
  \Edges(a,c)
  \Edges(a,p)
  \Edges(a,e)
  \Edges(a,f)
  \Edges(b,c)
  \Edges(b,p)
  \Edges(b,e)
  \Edges(b,f)
  \Edges(c,p)
  \AddVertexColor{red}{p}
\end{tikzpicture}
\end{center}
The configuration of guards above is a dominating set. Suppose that a problem arises at the red vertex. Each of the three guards is adjacent to the problem vertex. Moreover, if any guard moves to the problem vertex, the new configuration will still be a dominating set. Local information near the guards does not reveal that one movement is preferable. Suppose, though, that the guards move in the following way.

\begin{center}
\begin{tikzpicture}[scale=.5]
  \Vertex[x=0,y=0,NoLabel]{a}
  \Vertex[x=8,y=0,NoLabel]{b}
  \Vertex[x=4,y=-1,NoLabel]{c}
  \Vertex[x=4,y=-3]{p}
  \Vertex[x=4,y=-5,NoLabel]{e}
  \Vertex[x=4,y=-7,NoLabel]{f}
  \AddVertexColor{gray}{p}
  \AddVertexColor{gray}{b}
  \AddVertexColor{gray}{c}
  \Edges(a,b)
  \Edges(a,c)
  \Edges(a,p)
  \Edges(a,e)
  \Edges(a,f)
  \Edges(b,c)
  \Edges(b,p)
  \Edges(b,e)
  \Edges(b,f)
  \Edges(c,p)
\end{tikzpicture}
\end{center}
This is still a dominating set. It is also the case, however, that if a problem arises in one of the bottom two vertices it can not be addressed by the adjacent guard without losing domination.

\begin{center}
\begin{tikzpicture}[scale=.5]
  \Vertex[x=0,y=0,NoLabel]{a}
  \Vertex[x=8,y=0,NoLabel]{b}
  \Vertex[x=4,y=-1,NoLabel]{c}
  \Vertex[x=4,y=-3]{p}
  \Vertex[x=4,y=-5,NoLabel]{e}
  \Vertex[x=4,y=-7,NoLabel]{f}
  \AddVertexColor{gray}{p}
  \AddVertexColor{gray}{e}
  \AddVertexColor{gray}{c}
  \Edges(a,b)
  \Edges(a,c)
  \Edges(a,p)
  \Edges(a,e)
  \Edges(a,f)
  \Edges(b,c)
  \Edges(b,p)
  \Edges(b,e)
  \Edges(b,f)
  \Edges(c,p)
\end{tikzpicture}
\end{center}
The issue is that a `bad choice' was made by the guards earlier in the process. The initial problem  should have been addressed by the guard immediately above. Significantly, that information is not available to guards who only know simple local information about possible loss of dominance. Instead, it requires strategic, managerial thinking about the whole graph and the guards within it.

\section{Definitions}

The first definition, adjacency of dominating sets, makes the main definition clearer.

\begin{definition}
  Given a graph $G$ and dominating sets $S$ and $S'$, we say that $S$ and $S'$ are adjacent if they differ only by a pair of vertices adjacent in $G$: i.e. there are vertices $v\in S$ and $v' \in S'$ such that $(v,v')$ is an edge of $G$ and $S' = (S \setminus \{v\}) \cup \{v'\}$.
\end{definition}

It will be useful at times to use the term `legal guard move' to refer to the pair $(v,v')$ relating two adjacent dominating sets $S$ and $S'$.

Adjacent dominating sets have the same cardinality.

We now proceed to the main definition of the paper. The examples and discussion in the introduction illustrated the notion of domination of a graph subject to ongoing attack in which guards act via a simple protocol to coordinate their movements, but without a master planner. This definition captures the idea of a sufficient collection of autonomous guards.

\begin{definition}
  \label{autonomousFamilyDefinition}
  A collection $\mathcal{F}$ of subsets $S_i \subset V$ is said to be an autonomously dominating family if the following conditions are satisfied.
  \begin{enumerate}
  \item Each subset $S_i$ is a dominating set.
  \item For each subset $S_i$ and each vertex $v \in V \setminus S_i$, there is a subset $S_j \in \mathcal{F}$ which is adjacent to $S_i$ and contains $v$.
  \item For each subset $S \in \mathcal{F}$, every dominating set which is adjacent to $S$ is also in $\mathcal{F}$.
  \end{enumerate}
\end{definition}

The first condition simply records that the guard configuration at a given time is capable of responding to an attack at any vertex. The second condition is that there is always some guard that can respond to an attack without losing domination. The first and second conditions imply that each member of the family $\mathcal{F}$ is in fact secure dominating. The third is `closure under all reasonable movements of autonomous guards.' It says that it must be possible for any guard that can move to the attack without losing graph domination to do so. There is no coordinator who, with an eye to the whole, determines which guard moves. What this would mean in a practical situation is that a randomization would pick one guard to move, if multiple guards are free to do so.

\begin{remark}
Autonomous domination is about defending a graph in perpetuity. As a result the definition of autonomously dominating family does not privilege one of the many configurations as `initial.' Instead all possible arrangements of guards are treated equally. This is the reason for making the definition in terms of a family of subsets, rather than a statement about sequences. A definition using sequences could be made as well but would be more complicated. It is simpler to consider a whole family which is closed under the relation of adjacency for dominating sets.
\end{remark}

Having defined autonomous domination for a family, the notion for an individual set follows.

\begin{definition}
A set of vertices in a graph is said to be an autonomous dominating set if it belongs to an autonomous dominating family.  
\end{definition}

It is now possible to define a graph invariant, the autonomous domination number.

\begin{definition}
The autonomous domination number $\adn(G)$ of a graph $G$ is the least size of an autonomous dominating set.
\end{definition}

Autonomous domination differs from eternal and eternal foolproof domination in ways made evident by the form of the definition. In Definition \ref{autonomousFamilyDefinition}, one obtains eternal domination by dropping the third condition. All that is required is the existence of a possible guard (second condition). The master planner will determine which one moves. Eternal domination was introduced in \cite{IODG} and has been further studied, including in \cite{ESG} and~\cite{TB}. 

On the other hand, the third condition can be strengthened as follows.

\begin{quote}
  
\textbf{Strengthened Third Condition:} For each subset $S_i$, each vertex $v \in V \setminus S_i$, and each vertex $w \in S_i$, if $v$ and $w$ are adjacent, then $(S_i \setminus \{w\}) \cup \{v\}$ is also contained in $\mathcal{F}$.
\end{quote}

This gives foolproof eternal domination, in which guards move without considering whether such movement is prudent. The foolproof eternal domination number itself is very simple to compute (see Lemma \ref{boundedAboveByFoolproof}, derived from  \cite[Theorem 3]{IODG}). Variations have been investigated, such as the situation with multiple simultaneous guard movements \cite{FED}.

Many forms of domination are listed in the appendix of \cite{FDR}, a comprehensive reference on domination in graphs. Eternal domination and variants were introduced subsequent to its publication.

\section{Properties and Computational Methods}

\subsection{Elementary Bounds}

\label{sec:computationalMethods}
The foolproof eternal domination number is $n-\delta$ by \cite[Theorem 3]{IODG}, where $n$ is the order of the graph and $\delta$ is the minimal degree. It was noted above that foolproof eternal domination arises by strengthening the definition of autonomous domination, so that the following inequality is immediate.

\begin{lemma}
  \label{boundedAboveByFoolproof}
  Let $G$ be a graph with order $n$ and minimal degree $\delta$. Then $\adn(G) \le n - \delta$.
\end{lemma}

The upper bound can be attained. For example, the complete graphs attain the upper bound. 

It is also evident that the autonomous domination number is bounded below by the eternal domination number, since an eternal dominating set is subject to fewer conditions. In eternal domination a `manager' is able to have an eye to the whole of the network being defended, and this is reflected in the following inequality.

\begin{lemma}
  \label{boundedBelowByEternal}
  The autonomous domination number is bounded below by the eternal domination number.
\end{lemma}

\subsection{Some means of computation}

An adversary can attack at a series of independent vertices. Each such attack must be addressed by a separate guard, since the locations are independent. The next lemma records this.

\begin{lemma}
  \label{lemma:includeIndependent}
  Let $G$ be a graph and $\mathcal{F}$ an autonomous family of dominating sets. Given any independent set $I$ of vertices of $G$, there is a set $S$ in the automonous family such that $I \subset S$.  
\end{lemma}
\begin{proof}

  Let $S'$ be an element of $\mathcal{F}$. Let $S_0 = S' \cap I$, and define  $T = I \setminus S_0$. If $T$ is empty there is nothing to show. Suppose $T$ non-empty, and enumerate its elements $T = \{t_1,t_2,\dots,t_n\}$.

  By the definition of autonomous family, there is a dominating set $S_1$ adjacent to $S'$ which also contains $t_1$. By the independence of $I$, the guard moving to $t_1$ does not come from a vertex in $S_0$. Continuing in this manner, we find a sequence of dominating sets $S_i\in \mathcal{F}$ such that $S_i$ is adjacent to $\S_{i-1}$ and $|S_i \cap I| = |S_{i-1} \cap I| + 1$. Ultimately we arrive at $S_n$, which includes all vertices of $I$. This is the set $S$ that was sought.
\end{proof}

When considering concrete examples the following corollary is useful. When a graph contains an independent dominating set, it is easy to explore possible guard configurations.

\begin{corollary}
  \label{cor:supersetOfIndependent}
  Suppose that $G$ is a connected graph, $\mathcal{F}$ an autonomous dominating family, and $I$ an independent dominating set. Suppose further that the subgraph induced by the complement of $I$ is connected. Then any set $T$ of vertices of $G$  which contains $I$ and has the same cardinality as elements of $\mathcal{F}$ is contained in the autonomous family $\mathcal{F}$.
\end{corollary}
\begin{proof}
  By Lemma \ref{lemma:includeIndependent} there is some $S \in \mathcal{F}$ which contains $I$. Let the remaining guards move along paths from their initial locations to the remaining elements of $T$. These constitute a series of legal guard moves.
\end{proof}

A usful bound arises in the case that all dominating sets of a given size are secure dominating. 

\begin{theorem}
  \label{thm:secDomBound}
  If every dominating set of size $k$ is secure dominating, then $\adn \le k$.
\end{theorem}
\begin{proof}
  Define an autonomous family consisting of all dominating sets of size $k$.
\end{proof}

The converse does not hold. Consider a $K_3$ with two leaves from one vertex. The autonomous domination number of this graph is $3$, but the set consisting of the 3 $K_3$ vertices is dominating but not secure dominating.

\subsection{A partition count}

The eternal domination number is bounded above by a count of certain subcliques of a graph, given in \cite[Theorem 4]{IODG}. More specifically, if the vertices of a graph can be divided among $c$ subsets such that the induced graph on each subset is complete, then the eternal domination number is bounded above by $c$. 

This partition number $c$, though bounding the eternal domination number, does not provide a bound of $\adn$. Observe that a path $P_n$ can be partitioned (by taking adjacent pairs of vertices, with possibly one left over) into $\lceil \frac{n}{2} \rceil$ such subsets, all of whose induced graphs are complete. On the other hand, $\adn(P_n)$, computed in Proposition \ref{prop:adnPath}, exceeds this bound by an amount growing arbitrarily large as $n$ increases.

There is, however, the following technical proposition involving a partition of the set of vertices, which determines the autonomous domination number at least in some cases. It will be applied in Section \ref{sec:Families} to compute autonomous domination numbers of products of complete graphs.

\begin{proposition}
\label{partitionBound}
  Suppose that a graph $G$ is such that the vertices of $G$ may be partitioned into $k$ sets $S_i$ so that the induced subgraph on each $S_i$ is complete. Suppose further that $|S_i| > k$ for each $i$, and that there is at most one edge from a vertex in one partition $S_i$ to some other partition $S_j$. Then $\adn(G) = k$.  
\end{proposition}
\begin{proof}
  Let $S$ be a set of vertices of size $k$ containing one element of each of the sets $S_i$. This is a dominating set by the hypothesis that the induced graphs on the sets $S_i$ are complete.
  
  With guard configuration $S$, consider an attack at an arbitrary unoccupied vertex $v$. The vertex $v$ is contained in some $S_i$, and as such is defensible by the (single) guard also contained in the subset $S_i$. It is possibly adjacent to a guard at a vertex $w$, contained in a distinct partition set $S_j$, $j \ne i$. We claim that there is no dominating set adjacent to $S$ containing $w$.

There are at least $k+1$ vertices in $S_j$ by hypothesis. There are, in addition to the guard at $w$, $k-1$ other guards. If the guard at $w$ moves to $v$, the vertex $w$ is still guarded (by the guard that just moved). The $k$ or more other vertices of $S_j$ must be guarded. Since the $k-1$ guards are contained in partition sets other than $S_j$, those $k-1$ guards each defend at most one vertex of $S_j$, by the connectivity hypothesis. We see that the set $S$ is not adjacent to any dominating set which contains no element of $S_j$. Therefore each guard must remain in its respective partition set.

Thus the family generated by $S$ under sequences of legal guard moves is autonomous dominating.

Similar reasoning shows that no collection of $k-1$ vertices forms a dominating set. 
\end{proof}

\subsection{Chromatic Relation}

The eternal domination number of a graph is bounded above by the chromatic number of the graph complement, shown in \cite[Theorem 4]{IODG} and following discussion. The idea is this. Color vertices of the complementary graph, assign a guard to each color, and give instructions to each guard to address attacks only on vertices of its assigned color.

This chromatic bound does not suffice for autonomous domination. The order six graph in Section \ref{sec:basicExamples} has a three-colorable complement but its autonomous domination number is $4$.

\subsection{Connectedness}

Ordinarily it will be of interest only to consider connected graphs. The following lemma is straightforward.

\begin{lemma}
  \label{disjointUnion}
  The autonomous domination number of a disjoint union of graphs is the sum of the autonomous domination numbers of the summands.
\end{lemma}

\section{Elementary Examples}
\label{sec:elementaryExamples}
It is not difficult to check that $\adn(P_2) = 1$ and $\adn(P_3) = 2$. For longer paths, the following result gives the autonomous domination number.

\begin{proposition}
\label{prop:adnPath}
  The autonomous domination number of the path $P_n$ is $n-2$ for $n \ge 4$.
\end{proposition}

\begin{proof}
  
  First, consider $n\ge 5$. Let $\{a_1,a_2,\dots,a_n\}$ be the vertices of $P_n$, with edges $(a_i,a_{i+1})$. The set $I = \{a_2, a_4,\dots,a_{\lfloor \frac{n}{2} \rfloor}\}$ is independent and dominating, so by Lemma \ref{lemma:includeIndependent} it is a subset of an autonomous dominating set. Let $S$ be a dominating set with $|S|\le n-3$ which contains $I$. Let $j$ be an odd integer such that $a_j \in S$ and $a_{j+2} \notin S$ and $j+2\le n$. Then the legal guard moves $(a_{j+1},a_{j+2})$ followed by $(a_j,a_{j+1})$ show that the guards outside $I$ can accumulate among the vertices with higher indices. After a series of such moves, by the cardinality of $S$ we arrive at a dominating set which contains $\{a_2,a_4\}$ and which does not contain $\{a_1,a_3,a_5\}$.  Then $(a_4,a_5)$ is a legal guard move. The resulting set of vertices is no longer secure dominating, since only the guard at $a_2$ defends $a_3$, but $(a_2,a_3)$ is not a legal guard move since it would leave $a_1$ without defense.

In the case $n=4$, there is no dominating set of size $n-3 = 1$.
  
Having established the lower bound, now consider any dominating set $S$ of size $n-2$. Either the omitted vertices are adjacent, or they are not. If they are adjacent, they must be interior to the path (by the fact that $S$ is dominating), in which case $S$ is secure dominating. If they are not adjacent it is immediate that $S$ is secure dominating. By Theorem \ref{thm:secDomBound} and the lower bound already established, we see that $\adn(P_n) = n-2$.
  
\end{proof}

The foolproof eternal domination number of $P_n$ is $n-1$, and the eternal domination number is $\lceil \frac{n}{2} \rceil$ as computed in \cite{IODG}, so that the autonomous domination number is intermediate, but closer to the foolproof for large~$n$.

A similar relation is present in the case of cycles. For small cycles, it is straightforward to check that $\adn(C_3) = 1$, $\adn(C_4) = 2$, and ${\adn(C_5)=3}$. The general result is the following.

\begin{proposition}
  The autonomous domination number of the cycle $C_n$ is $n-3$ for $n \ge 6$.
\end{proposition}

\begin{proof}
  Let $n$ be at least $8$, and let $\{a_1,\dots,a_n\}$ be the vertices of $C_n$, with edges $(a_i,a_{i+1})$ and $(a_1,a_n)$. Let $I$ be the set of all odd-indexed vertices with index strictly less than $n$. Let $S$ be any dominating set of size $|S| \le n-4$ which contains $I$. Reasoning as in the proof of Proposition \ref{prop:adnPath}, $S$ is connected through a series of legal guard moves to a dominating set which contains the vertices of $I$ as well as the vertex $a_n$ and all vertices with sufficiently large even index. In particular, by cardinality, the vertices $\{a_2,a_4,a_6, a_8\}$ are unguarded. The series of legal guard moves $(a_3,a_2)$ and $(a_7,a_8)$ yields a dominating set which is not secure dominating, since the guard at $a_5$ must defend both $a_4$ and $a_6$ but cannot do so.

In the case of $n=6$, there is no secure dominating set of size $2$. In the case of $n=7$, the independent set $\{a_1,a_4,a_6\}$ is dominating and so would necessarily be contained in an autonomous family whose elements have size $3$. But $(a_6,a_7)$ is a legal guard move, yielding $\{a_1,a_4,a_7\}$ which is not secure dominating, since only $a_4$ defends $a_3$ and $a_5$.

Every dominating set of size $n-3$ is secure dominating (there is no room to end up with an `isolated guard'), so by Theorem \ref{thm:secDomBound} we conclude that $\adn(C_n) = n-3$ for $n\ge 6$.
\end{proof}

The comparable results are that the eternal domination number of $C_n$ is $\lceil \frac{n}{2} \rceil$ and the foolproof eternal domination number is $n-2$, again given in \cite{IODG}.

\section{Autonomous Domination Numbers for Certain Families of Graphs}
\label{sec:Families}

Let $p \le q$ be natural numbers. In \cite[Proposition 2]{IODG} it is shown that the eternal domination number of $K_p \times K_q$ is $p$, while the foolproof eternal domination number is $pq-(p+q)+2$, so that the two differ widely in general. For this family of graphs the autonomous domination number happens to coincide with the eternal domination number.

\begin{proposition}
  Suppose $p$ and $q$ are natural numbers with $p \le q$. Then $\adn(K_p \times K_q) = p$.
\end{proposition}
\begin{proof}
  First consider the case that $p < q$. Partition the graph $K_p \times K_q$ into the $p$ sets of vertices of the copies of $K_q$, each of which has size $q$. Then apply Proposition \ref{partitionBound} to infer that $\adn(K_p \times K_q) = p$.

  It remains to treat the case $p=q$. By \cite[Proposition 2]{IODG} the eternal domination number is $p$, so the elementary bound of Lemma \ref{boundedBelowByEternal} implies that $\adn$ is at least $p$. We now show that $p$ guards suffice.
  
  Enumerate the vertices of each $K_p$ with non-negative integers $\{0,\dots,p-1\}$. The vertices of the product $K_p \times K_p$ are then ordered pairs $(i,j)$ of such integers.

  Distribute $p$ guards to $p$ vertices with distinct first coordinates; the second coordinate of each location is arbitrary. This is a dominating set by the completeness of the factor $K_p$.
  
  Consider a legal guard move whose result is that there is an integer $i$ such that there is no guard at $(i, j)$ for all $j$. In other words, the collection of vertices $(i, \bullet)$ in $K_p \times K_p$ has been vacated. For this to be the case, there must have been guards at $(j_k, k)$ for all $k$ (the first coordinates need not be distinct) so that every complete subgraph of the form $(i, \bullet)$ remains dominated.

  Therefore, at every time, legal guard movements will be such that either every guard has a distinct first coordinate, or every guard has a distinct second coordinate, and all such subsets of vertices are dominating sets.

Since the eternal domination number $p$ is a lower bound for the autonomous domination number, we conclude that $\adn(K_p\times K_p) =p$.
\end{proof}

Ladder graphs are also an instance in which eternal and foolproof eternal domination numbers diverge widely. In this case the autonomous domination number differs negligibly from the foolproof one. By \cite[Theorem 8]{IODG} the eternal domination number of $P_2 \times P_n$ is $n$, and the foolproof eternal domination number is $2n-2$.

\begin{proposition}
The ladder graph $P_2 \times P_n$ has autonomous domination number ${\adn(P_2 \times P_n) = 2n-3}$.
\end{proposition}

\begin{proof}
  Let the vertices of $P_2\times P_n$ be $\{a_1,\dots,a_n,b_1,\dots,b_n\}$ with edges $(a_i,a_{i+1})$, $(b_i,b_{i+1})$, and $(a_i,b_i)$. Let $S$ be a dominating set of size $n \le |S| \le 2n-4$ which contains the independent dominating set $\{a_1,b_2,a_3,b_4,\dots\}$. Suppose that $S$ contains $a_{2k}$ and does not contain $a_{2k+2}$. The pairs $(a_{2k+1},a_{2k+2})$ and $(a_{2k},a_{2k+1})$ are legal guard moves whether or not $b_{2k+1}$ is occupied. Thus excess guards in the $a$-subgraph can be moved to vertices of higher index. The same can be done with guards among the $b$ subgraph vertices. After many such moves, we obtain a dominating set $S'$ connected to $S$ through a series of legal guard moves does not contain $\{a_2,a_4,b_1,b_3\}$. Then $(b_2,b_3)$ and $(a_3,a_4)$ are legal guard moves, but the resulting dominating set is not secure dominating since the guard at $a_1$ is isolated and cannot defend both $b_1$ and $a_2$.

  To show that $2n-3$ autonomous guards suffice, consider any dominating set of size $2n-3$. All vertices but those on the ends of the ladder have degree 3, which means that an adjacent vertex is necessarily occupied simply by cardinality. Since the set is presumed dominating, either at least one of $\{a_1,b_1\}$ is included or both $a_2$ and $b_2$ are included. In either case the set is secure dominating. The same reasoning applies for the other end of the ladder.
\end{proof}

\section{Counterexamples}
\label{sec:Counterexamples}

When people act with limited information, circumstances that are in principle advantageous can nonetheless yield detrimental results. The following two examples show that autonomous domination reflects this, in failure that arises through an excess either of choices or of strength. The addition of an edge, which offers new ways for guards to defend locations, can lead them away from the most strategically suitable positions. Additional guards, which in theory make a defensive arrangement stronger, can allow for poor positioning due to temporary security arising from the abundance of defenders.

The examples below are evidence of the utility of the concept of autonomous domination. The parameter does not decrease monotonically with edge addition, and it is not a super-hereditary. These suggest that autonomous domination is a good abstraction that can account for difficulties that arise when agents act tactically with restricted information.

\subsection{Edge addition}

The autonomous domination number can decrease when an edge is added. A simple example is completing a path on three vertices to the complete graph on these three vertices, in which case the autonomous domination number decreases from 2 to 1.

Adding an edge can increase the autonomous domination number, however. A new edge means that guards can be drawn away from one portion of the graph and into another, and this could leave the source portion vulnerable.

\begin{lemma}
  The graph on the left in Figure~\ref{fig:House} has autonomous domination number $2$.
\end{lemma}
\begin{proof}
The vertices of the graph separate into two induced $K_3$ subgraphs, the lower and the upper levels, and Proposition \ref{partitionBound} applies. \end{proof}
\begin{figure}
  \caption{Edge addition can increase the autonomous domination number}
  \label{fig:House}
  \begin{framed}
  \includegraphics[scale=.9,trim=-20 0 -10 0, clip]{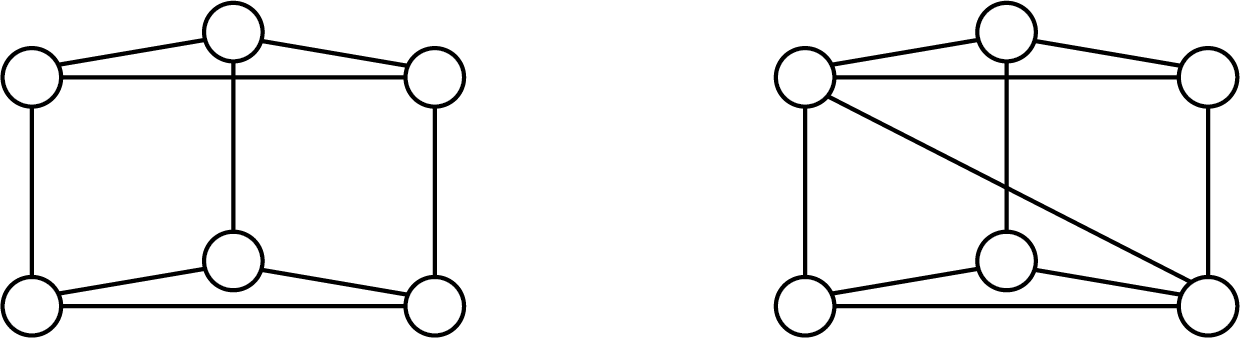}
\end{framed}
\end{figure}

The additional diagonal edge added to obtain the graph on the right causes the autonomous domination number to increase.

\begin{lemma}
  The autonomous domination number of the graph on the right in Figure~\ref{fig:House} is $3$.  
\end{lemma}
\begin{proof}
The autonomous domination number must be at least $2$, since the graph has an independent set of size $2$. Consider the independent dominating set consisting of the upper left and lower central vertices. This is adjacent to the dominating set consisting of the upper left and upper central vertices. But this is not secure dominating since an attack on the upper right vertex cannot be safely covered.
\end{proof}

 Another simple example is adding an edge to connect the disjoint union of $C_5$ and $K_3$. The graph $C_5 \coprod K_3$ has autonomous domination number 5, by the additivity statement Lemma \ref{disjointUnion}. When a bridge is built between them, a guard can be `siphoned off' of the $C_5$ subgraph in a way that leaves it exposed to attack. The autonomous domination number of the new graph is 6, rather than 5.

\subsection{Failure through excess of guards}

In can be that $n$ guards suffice for autonomous domination but $n+1$ guards do not. Thus autonomous domination is not a super-hereditary property.

\begin{definition}
  \label{def:houseWithConnection}
  Define a graph with the vertices $$\{a_1,a_2,a_3,a_4,b_1,b_2,b_3,b_4,b_5\}$$ and the following edges.
  \begin{itemize}
  \item Edges such that the set $\{a_i\}_{i=1}^4$ induces a complete subgraph.
  \item Edges such that the set $\{b_i\}_{i=1}^5$ induces a complete subgraph.
  \item $(a_i,b_i)$ for $1 \le i \le 4$.
  \item $(b_1,a_2)$
  \item There are no other edges.
  \end{itemize}
\end{definition}

The graph is depicted in Figure~\ref{fig:HouseWithConnection}. The vertices $a_i$ are below and the vertices $b_i$ are above. The topmost vertex is $b_5$, not adjacent to any of the~$a_i$.

\begin{figure}
  \caption{The graph of Definition \ref{def:houseWithConnection}}
  \label{fig:HouseWithConnection}
  \begin{framed}
    \includegraphics[trim=-70 0 -70 0, clip]{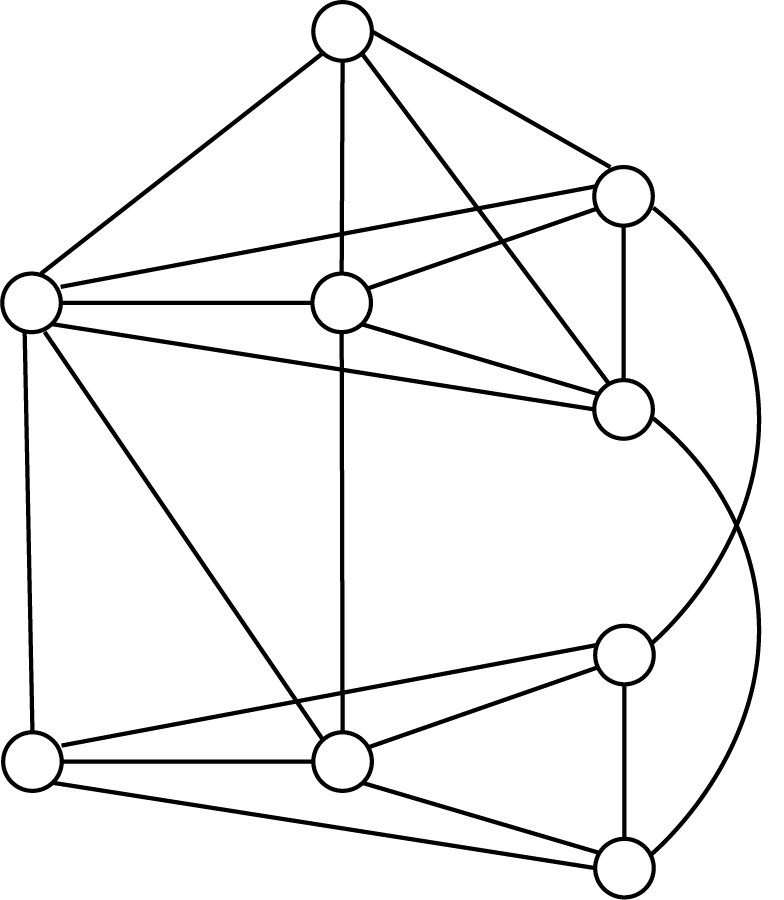}
  \end{framed}
\end{figure}

The eternal domination number of this graph is $2$. This is also the autonomous domination number. Yet there is no autonomous set with three guards. The reason is that an arbitrary dominating set is connected through a series of legal guard moves with $\{b_1,a_3,a_4\}$, which is connected by a pair of legal guard moves with $\{b_1,b_3,b_4\}$. This last set is not secure dominating, since it is not adjacent to a dominating set containing $b_5$. Informally, with three guards it is possible that they all move `upstairs,' and this is a bad decision.

Because autonomous domination is not super-hereditary, it does not suffice to show that a certain number of guards fail in order to find a lower bound on the autonomous domination number. Instead, every possibility greater than or equal to the eternal domination number must be considered.

\section{Realizability of Parameter Values}

It is straightforward to observe that the eternal domination number is bounded below by the domination number, and similarly that the autonomous domination number is bounded below by the eternal domination number (Lemma \ref{boundedBelowByEternal}). A question following from these bounds is whether every possible collection of three natural numbers satisfying the inequalities constitute the domination number, eternal domination number, and autonomous domination number of a graph. This is so (excluding a trivial exception), as will be shown below. The proof will require reference to a number of families of graphs constructed for the purpose. These will be presented first, and their various parameters computed. The theorem then collects these results.

\begin{definition}
  Given a natural number $n\ge 2$ define a graph $\Acal_n$ with vertices
  $$\{a_1, a_2, \dots, a_{2n+1},b_1, b_2, \dots, b_{2n}, c\}$$ and  the following edges.
  \begin{itemize}
  \item  The vertices $a_i$ induce a complete subgraph.
  \item  The vertices $b_i$ induce a complete subgraph.
  \item $(a_i,b_i)$ for all $1 \le i \le 2n$
  \item $(a_i,b_j)$ for all $1 \le i \le n$ and $1 \le j \le 2n - i$ (observe that some of these were included in the previous item)
  \item $(c,a_i)$ fo all $1 \le i \le 2n+1$
  \item $(c,b_j)$ fo all $1 \le j \le 2n$   
  \end{itemize}
\end{definition}

The graph $\Acal_2$ is sketched in Figure \ref{fig:ConeOnHouse}. Some simplification has been made there to avoid clutter. The ovals denote induced complete subgraphs, and the vertex $c$, adjacent to all other vertices, is shown with just a few half edges.

\begin{figure}
  \centering
  \caption{The graph $\Acal_2$}
  \label{fig:ConeOnHouse}
  \begin{framed}
 \includegraphics[trim=-40 0 -40 0, clip]{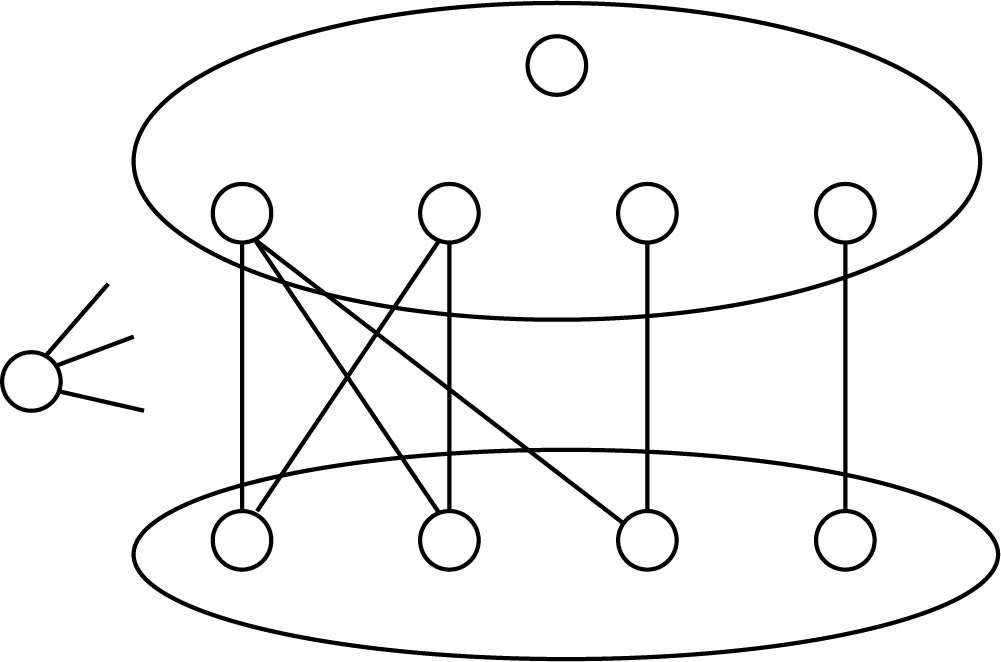}
\end{framed}

\end{figure}

\begin{proposition}
  The parameters of $\Acal_n$ are these: ${\gamma(\Acal_n)=1}$, ${\edn(\Acal_n)=2}$, and  ${\adn(\Acal_n)=n+1}$.
\end{proposition}
\begin{proof}
  The vertex $c$ is adjacent to every other vertex, so $\gamma(\Acal_n)=1$. The vertices $a_{2n+1}$ and $b_{1}$ are not adjacent, so $\edn(\Acal_n)\ge 2$. That two guards suffice follows from observing that $\{c,a_1,a_2,\dots,a_{2n+1}\}$ and $\{b_1, b_2, \dots, b_{2n}\}$ induce complete subgraphs.

  Consider any natural number $g$ satisfying $2 \le g \le n$. There is no autonomous dominating set of size $g$. Any dominating set is adjacent to one which includes $c$. Since $c$ is adjacent to every vertex, we see that any two sets of the same size, both of which contain $c$, are adjacent through a series of legal guard moves. An arbitrary dominating set of size $g$ is thus connected to $\{a_{2n--g+2}, a_ {2n-g+3} \dots, a_{2n},c\}$. This set is adjacent to $\{a_{g-1},a_{2n-g+2}, a_ {2n-g+3} \dots, a_{2n}\}$, which is a dominating set. It is not secure dominating, however, since an attack on $a_{2n+1}$ leads to failure. 

  It remains to show that $\Acal_{n}$ has an autonomous dominating set of size $n+1$. Any dominating set of $n+1$ vertices among the $a_i$ is secure dominating. No set consisting only of $b$ vertices is dominating, since none is adjacent to $a_{2n+1}$. Thus the only remaining case to check is that a dominating set including $c$ is secure dominating. It follows immediately from the existence of induced complete subgraphs that a dominating set including both $a$ and $b$ vertices, as well as $c$, is secure dominating. The set $\{c,b_1,b_2,\dots,b_{2n}\}$ is also secure dominating, since it is dominating, and exchanging any element with an adjacent $a$ vertex leads to a set containing both $a$ and $b$ vertices. Such a set is dominating.
\end{proof}

\begin{definition}
  Given a non-negative integers $m$ and $n$, define the graph $\Bcal_{m,n}$ with vertices
$$\{a_1,a_2,\dots a_{2n+3},b_1,b_2,\dots,b_{2n+4},c^1_1,c^2_1,c^1_2,c^2_2,\dots,c^1_m,c^2_m\}$$ and the following edges.
  \begin{itemize}
  \item  The vertices $a_i$ induce a complete subgraph.
  \item  The vertices $b_i$ induce a complete subgraph.
  \item $(a_i,b_i)$ for all $1 \le i \le n$
  \item $(c^1_i,c^2_i)$ for all $1 \le i \le m$
  \item $(b_{n+1},c^1_i)$ for all $1 \le i \le m$
  \item $(a_1,b_2)$ and $(a_2,b_1)$
  \item $(b_{3+i},a_{3+j})$ for each $0 \le i \le n$ and $0 \le j \le 2n - i$
  \item There are no other edges.
  \end{itemize}
\end{definition}

A sketch of such a graph is in Figure \ref{fig:HouseWithSpikes}. There the parameter values are $m=2$ and $n=0$. 

\begin{figure}
\centering
\caption{The graph $\Bcal_{2,0}$}
\label{fig:HouseWithSpikes}
\begin{framed}
  \includegraphics[trim=-25 0 -25 0, clip]{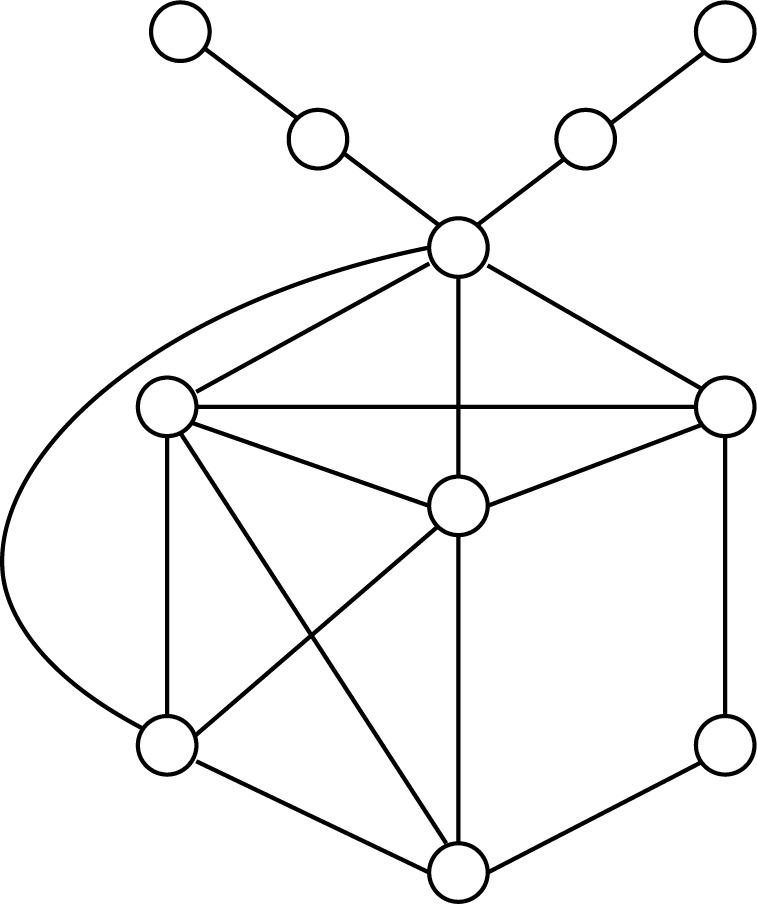}
\end{framed}
\end{figure}

\begin{proposition}
The domination parameters of $\Bcal_{m,n}$ are:  ${\gamma(\Bcal_{m,n})=\edn(\Bcal_{m,n}) = m+2}$, and ${\adn(\Bcal_{m,n})=m+n + 3}$.
\end{proposition}

\begin{proof}

   The set $I = \{a_{2n+3},b_{1},c^2_1,c^2_2,\dots,c^2_m\}$ is an independent dominating set of size $m+2$. There is no smaller dominating set, since at least one guard is needed for each of the $m$ branches $\{c^1_i,c^2_i\}$ and none of these will guard $b_i$ for $i \le 2n+3$ or any vertex among the $a_j$.

  The eternal domination number is also $m+2$ by the clique bound of \cite[Theorem 4]{IODG}, since one guard can be allocated to each of the $m$ pairs $\{c^1_i,c^2_i\}$, one can be allocated to the vertices $a_i$, and finally one to the vertices $b_j$.

  We now show that the autonomous domination number is not less than $m+n + 2$. Let $k$ be a natural number satisfying $m+2 \le m+1+k < m+n+3$, and let $m+1+k$ be the size of some dominating set $S$ of vertices which contains the independent set $I$. Let $T$ be the (possibly empty) set of $k-1$ vertices $\{b_{2n+3 - k + 2},b_{2n+3-k+3}, \dots, b_{2n+3}\}$. Since the graph induced by the complement of $I$ is connected, Corollary~\ref{cor:supersetOfIndependent} implies that there is a series of legal guard moves connecting $S$ with the dominating set $S' = \{a_{2n+3}, b_1, c_1^2,c_2^2,\dots,c_m^2\} \cup T$.

  Let $S'$ and $S''$ be adjacent by the legal move $(a_{2n+3},b_{2+k})$. Observe that $S''$ is dominating, since each element of $T$ is adjacent to exactly one of $\{a_i\}_{i=2n+3 - k + 2}^{2n+3}$ and $b_{2+k}$ is adjacent to each of $\{a_3, a_4, \dots, a_{2n+3-k+1}\}$.

The set $S''$ is not, however, secure dominating. Given an attack at $b_{2n+1}$, no guard can cover without losing domination. This shows that the autonomous domination number is greater than $m+n+2$.

  Now we must show that $m+n+3$ guards suffice for autonomous domination. Consider any dominating set of this size. It necessarily contains at least $m$ vertices from among $c^j_i$. It follows that the set contains at most $n+3$ vertices from $\{a_i\}$ and $\{b_j\}$. For the sake of being a dominating set, we also know that there are at least two such vertices.
  
  There are a few cases to consider.

  \textit{Case 1: Guards are at $a$ and $b$ vertices.}
  This case is straightforward. Since $\{a_i\}$ and $\{b_i\}$ are each cliques, the set is secure dominating.

  \textit{Case 2: No guard is at any $b$ vertex.}
  In order for the set to be dominating, it necessarily contains the $n$ vertices $\{a_{n+4}, a_{n+5}, \dots, a_{2n+3}\}$ (each of which is the unique vertex among the $\{a_i\}$ adjacent to the similarly indexed $b_i$). The set must in addition contain at least one of $\{a_1,a_2\}$.

  Either some guards are at $a$ vertices and some are at $b$ vertices, in which case the set is secure dominating, or all guards are on $a$ vertices. Only the latter case deserves further consideration. An attack at a $b$ vertex can be covered, by domination and the fact that the $a$ and $b$ vertices induce respective complete graphs. An attack at an $a$ vertex is securely coverable, which is seen as follows. Let $a_i$ be the first occupied $a$ vertex (counting up from index $1$). Then all guards in the vertices from $a_{i+1}$ to $a_{2n-i}$ are redundant. That means that at most the vertices from $a_{2n-i+1}$ up to $a_{2n}$ are additional necessary guards. That is $i$ total. By counting, $i$ is at most $n$. (The situation $i=n+1$ is impossible, since then the occupied vertices would be $n+1$ through $2n+1$, and this is not a dominating set.)
\end{proof}

\begin{definition}
  Given a natural number $m\ge 2$ and a natural number $n$, define the graph $\Ccal_{m,n}$ with vertices are $$\{a_1 a_2, b_1, b_2, \dots,b_m, c_1, c_2, \dots, c_n\}$$ and the following edges.
  \begin{itemize}
  \item Each vertex $a_i$ ($i=1,2$) is adjacent to every other vertex.
  \item The set of vertices $\{c_i\}_{i=1}^n$ induces a complete subgraph.
  \item There are no other edges.
  \end{itemize}
\end{definition}

The graph $\Ccal_{3,3}$ is depicted in Figure \ref{fig:FattenedStar}.

\begin{figure}
  \caption{The graph $\Ccal_{3,3}$}
  \centering
  \label{fig:FattenedStar}
\begin{framed}
 \includegraphics[trim=-40 0 -40 0, clip]{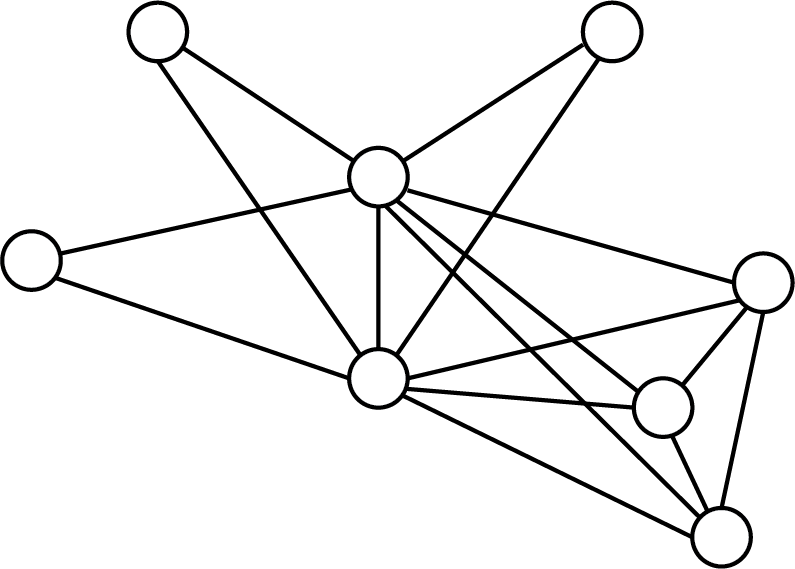}
\end{framed}
\end{figure}

\begin{proposition}
  The domination parameters of $\Ccal_{m,n}$ are these: ${\gamma(\Ccal_{m,n}) = 1}$, ${\edn(\Ccal_{m,n}) = m + 1}$, and ${\adn(\Ccal_{m,n}) = m+n}$.
\end{proposition}
\begin{proof}
  Since $a_1$ is adjacent to every vertex the domination number is 1.

  The set $\{b_1, b_2, \dots. b_m, c_1\}$ is an independent set, so that the eternal domination number is at least $m+1$. In fact $m+1$ suffice. Let one guard be assigned to $\{b_1, a_1, a_2\}$, and let one guard be assigned to each vertex $b_i$ ($i>2$). Finally, let a guard be assigned to $\{c_1, c_2, \dots , c_n\}$.

  Let any number $k$ of guards from $m +1$ to $m+n-1$ be given, and consider any dominating set of size $k$. Such a dominating set is adjacent to one containing the vertex $a_1$. Then, via a series of legal guard moves, we arrive at the dominating set $\{a_1, c_1, c_2, \dots, c_n, b_1, b_2 , \dots , b_{k-n-1}\}$. This is not a secure dominating set, since an attack at $b_{k-n}$ can only be covered by the guard at $a_1$, and then no guard is adjacent to $b_{k-n+1}$.
  
A set with $m+n$ vertices either contains all the peripheral vertices $b_i$ and $c_j$, or it contains one of the central vertices $a_k$. In either case such a set is secure dominating.
  
\end{proof}

\begin{definition}
  Given natural numbers $m$ and $n$, define the graph $\Dcal_{m,n}$  with vertices $$\{a_1, a_2, b_1^1, b_1^2, b_2^1, b_2^2, \dots , b_m^1, b_m^2, c_1, c_2, c_3, \dots , c_n\}$$ and the following edges.
  \begin{itemize}
  \item The set of vertices $\{a_1, a_2, b_1^1, b_2^1, \dots, b_m^1\}$ induces a complete subgraph.
  \item $(b_i^1,b_i^2)$ for each $i$, $1 \le i \le m$.
  \item The set of vertices $\{a,_2,c_1, c_2, \dots , c_n\}$ induces a complete subgraph.
  \item There are no other edges.
  \end{itemize}
\end{definition}

The graph $\Dcal_{3,3}$ is depicted in Figure \ref{fig:ManyLeavesOneFat}. The oval encloses vertices inducing a complete subgraph.
\begin{figure}
\caption{The graph $\Dcal_{m,n}$}
  \centering
  \label{fig:ManyLeavesOneFat}
  
\begin{framed}
 \includegraphics[ trim=-40 0 -40 0, clip]{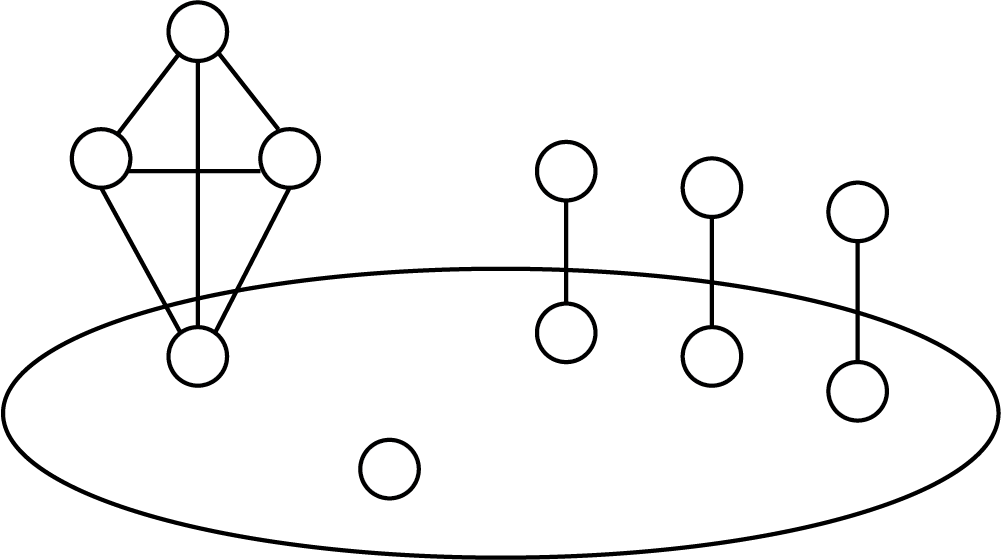}
\end{framed}
\end{figure}

\begin{proposition}
  The domination parameters of $\Dcal_{m,n}$ are these: $\gamma(\Dcal_{m,n}) = m+1$, $\edn(\Dcal_{m,n}) = m+2$, and $\adn(\Dcal_{m,n}) = m+n+1$.
\end{proposition}
\begin{proof}
  The vertices $\{b_1^1, b_2^1, \dots, b_m^1, a_2\}$ form a dominating set. No smaller dominating set exists. At least $m$ guards are required for the $m$ leaves $b_i^2$, and another is required for vertices $c_i$.

  The set $\{a_1,b_1^2, b_2^2, \dots, b_m^2, c_1\}$ is an independent set of size $m+2$, giving a lower bound for the eternal domination number. In fact $m+2$ guards suffice, since $\{a_1,a_2\}$, $\{\{b_i^1, b_i^2\}_{i=1}^m\}$, and $\{c_i\}_{i=1}^n$ is a partition of the vertices into $m+2$ subsets all inducing complete subgraphs.

  Let $k$ be a number of guards $m+2 \le k \le m+n$. Consider a dominating set of size $k$ containing the independent set of vertices $\{a_1, b_1^2, b_2^2, \dots , b_m^2, c_1\}$. This set has size $m+2$, so that there are $k-m-2$ remaining guards. Then via a series of legal guard moves we obtain the set $$\{a_1,  b_1^2, b_2^2, \dots , b_m^2, c_1, c_2, \dots, c_{k-m-1}\}$$ which is adjacent to $$\{a_1,  b_1^1, b_2^2, \dots , b_m^2, c_1, c_2, \dots, c_{k-m-1}\}$$ which is adjacent to $$\{a_2,  b_1^1, b_2^2, \dots , b_m^2, c_1, c_2, \dots, c_{k-m-1}\}$$ which is adjacent to $$\{  b_1^1, b_2^2, \dots , b_m^2, c_1, c_2, \dots, c_{k-m}\}$$ which is not secure dominating, since an attack at $b_1^2$ is only covered by the guard at $b_1^1$, leaving $a_1$ unprotected.
\end{proof}

\begin{definition}
  Let $m$ and $n$ be natural numbers. Define the graph $\Ecal_{m,n}$ with vertices  $$\{a_1, a_2, a_3, \dots, a_{m+3}, b_1, b_2, c_1, c_2, \dots , c_n\}$$ and the following edges.
  \begin{itemize}
  \item The vertices $a_i$ induce a complete subgraph.
  \item $(a_i,b_j)$, for each pair $(i,j) \in \{1,2\}\times \{1,2\}$.
  \item $(a_{m+3}, c_i)$ for each $i \in \{1, 2, \dots, n\}$.
  \end{itemize}
  
\end{definition}

The graph $\Ecal_{3,3}$ is depicted in Figure~\ref{TwoBridgeWithLeaves}. The vertices $a_i$, inducing a complete subgraph, are enclosed by the oval and their edges are omitted.r

\begin{figure}
  \caption{The graph $\Ecal_{3,3}$}
  \centering
  \label{TwoBridgeWithLeaves}
\begin{framed}
 \includegraphics[trim=-40 0 -40 0, clip]{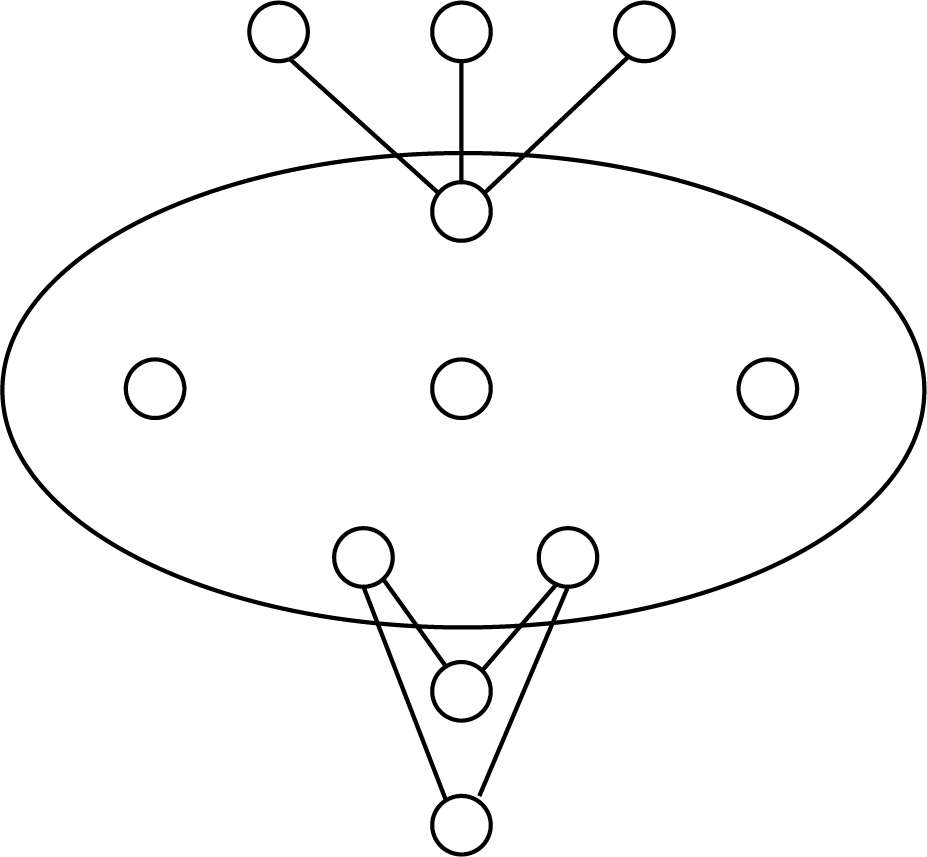}
\end{framed}
\end{figure}

\begin{proposition}
  The domination parameters of $\Ecal_{m,n}$ are these: ${\gamma(\Ecal_{m,n}) = 2}$, ${\edn(\Ecal_{m,n}) = n+3}$, and ${\adn(\Ecal_{m,n}) = m+n+3}$.
\end{proposition}
\begin{proof}
  There is no single vertex adjacent to all vertices, but each vertex is adjacent to either $a_1$ or $a_{m+3}$. This establishes the domination number.

  The set $\{b_1, b_2, a_3, c_1, c_2, \dots, c_n\}$ is an independent set of size $n+3$, establishing a lower bound for the eternal domination number. That $n+3$ guards suffice is seen by noting that stationary guards can be assigned to each $b_i$ and $c_j$, and one final guard is needed for the complete graph induced by $\{a_i\}$.

  Let $k$ be a natural number $n+3 \le k \le m + n + 2$ and consider a dominating set of size $k$ which contains the independent set $\{b_1, b_2, a_3, c_1, c_2, \dots, c_n\}$. Through a series of legal guard moves we obtain a dominating set which contains $b_1$ and $b_2$ but neither $a_1$ nor $a_2$ nor $a_3$. This is adjacent to one containing $a_1$ but not $b_1$, and this is adjacent to one containing $a_2$ but not $b_2$. This is adjacent to the dominating set containing $a_3$ but not $a_2$. This set is not secure dominating. An attack at one of the vertices $b_i$ can only be addressed by the guard at $a_1$, at which point the other $b$ vertex is left without defense.

That $m+n+3$ guards suffice follows from counting. With $m+n+3$ guards at least two of $\{a_1,a_2,b_1,b_2\}$ will be occupied, hence neither $b_1$ nor $b_2$ will be unguarded.  
\end{proof}

\begin{definition}
  Given natural numbers $\ell$, $m$, and $n$, define the graph $\Fcal_{\ell,m,n}$ with vertices $$\{a_1, a_2, \dots, a_\ell, b_1, b_2, \dots, b_m, c_1, c_2, c_3,\dots, c_{\ell+n}\}$$ and the following edges.
  \begin{itemize}
  \item $(a_i,c_i)$ for each $i \in \{1,2,\dots,\ell\}$.
  \item $(b_i,c_j)$ for each $i \in \{1,2,\dots,m\}$ and each $j \in \{1,2,\dots,\ell\}$.
  \item The vertices $c_i$ induce a complete subgraph. 
  \item There are no other edges.
  \end{itemize}
  
\end{definition}

The graph $\Fcal_{4,3,3}$ is depicted in Figure~\ref{fig:ManyLeavesAndBridges}. The vertices $c_i$, inducing a complete subgraph, are enclosed in the oval and their edges are omitted.

\begin{proposition}
  The domination parameters of $\Fcal_{m,n}$ are these: ${\gamma(\Fcal_{\ell,m,n}) = \ell}$, ${\edn(\Fcal_{\ell,m,n}) = \ell+ m + 1}$, and ${\adn(\Fcal_{\ell,m,n}) = \ell + m + n}$.
  
\end{proposition}
\begin{proof}
  One vertex of each pair $\{a_i,c_i\}$ must be occupied in order to obtain a dominating set, so that the domination number is at least $\ell$. Provided that at least one guard is stationed at $c_i$, the set is dominating, since each vertex $c_i$ is adjacent to every $b$ and $c$ vertex.

  The set $\{a_1,a_2,\dots,a_\ell,b_1,b_2,\dots,b_m,c_{\ell+1}\}$ is an independent set of size $\ell + m +1$, which gives a lower bound for the eternal domination number. That many guards suffice, since stationary guards can be assigned to each $a_i$ and $b_j$ and one guard can defend the remaining vertices $c_i$ which induce a complete subgraph.

  Let $k$ be a natural number $\ell +m+1 \le k \le \ell +m +n - 1$. Consider a dominating set of size $k$ which contains the independent set $\{a_1,a_2,\dots,a_\ell,b_1,b_2,\dots,b_m,c_{\ell+1}\}$. This is adjacent to one with $c_1$ replacing $a_1$. Then all the guards on the $b_i$ can move to $c_i$ without losing domination. In the end, since the $\ell$ leaves are occupied, as well as interior vertices $\{c_{\ell+1},\dots,c_{\ell + n}\}$, then because $k \le \ell + m + n - 1$ we see that at most $m-1$ of the vertices ${c_1,\dots,c_\ell}$ are occupied. There is then no defense against the series of attacks $b_1$, $b_2$, \dots $b_m$.

  We finally show that $\ell +m +n$ guards suffice for autonomous domination. Consider any dominating set of that size. Since the set is dominating, at least one member of each pair $\{a_i,c_i\}$ is contained in the set. Moreover, either no $c_i$ is in the set, in which case all of the vertices $b_i$ are, or at least one of the $c_i$ is in the set. In the former case secure domination is straightforward: every vertex $a_i$ is occupied, and every vertex $b_i$ is occupied, and there are $n>0$ guards remaining necessarily contained in the set $\{c_i\}$ which induces a complete subgraph. In the latter case, let $j$ be the number of vertices $c_i$ contained in the dominating set $1 \le j \le \ell$. Then $\ell-j$ guards occupy leaves, since the set is dominating. At most $n$ guards are contained in the subset of vertices $\{c_{\ell+1},\dots,c_{\ell + n}\}$, which means at least $\ell+m+n - (j + (\ell-j) +n) = m$ guards are contained in the vertices $b_i$. But there are only $m$ such vertices. Thus the only empty vertices are either among pairs $\{a_i,c_i\}$ (which contain a guard by hypothesis) or are among the vertices $\{c_{\ell+1},\dots,c_{\ell+n}\}$ (which induce a complete subgraph and contain a guard). Thus the set is secure dominating.
\end{proof}

\begin{figure}
  \caption{The graph $\Fcal_{4,3,3}$}
  \centering
  \label{fig:ManyLeavesAndBridges}
\begin{framed}
 \includegraphics[scale=.8,trim=-40 0 -40 0, clip]{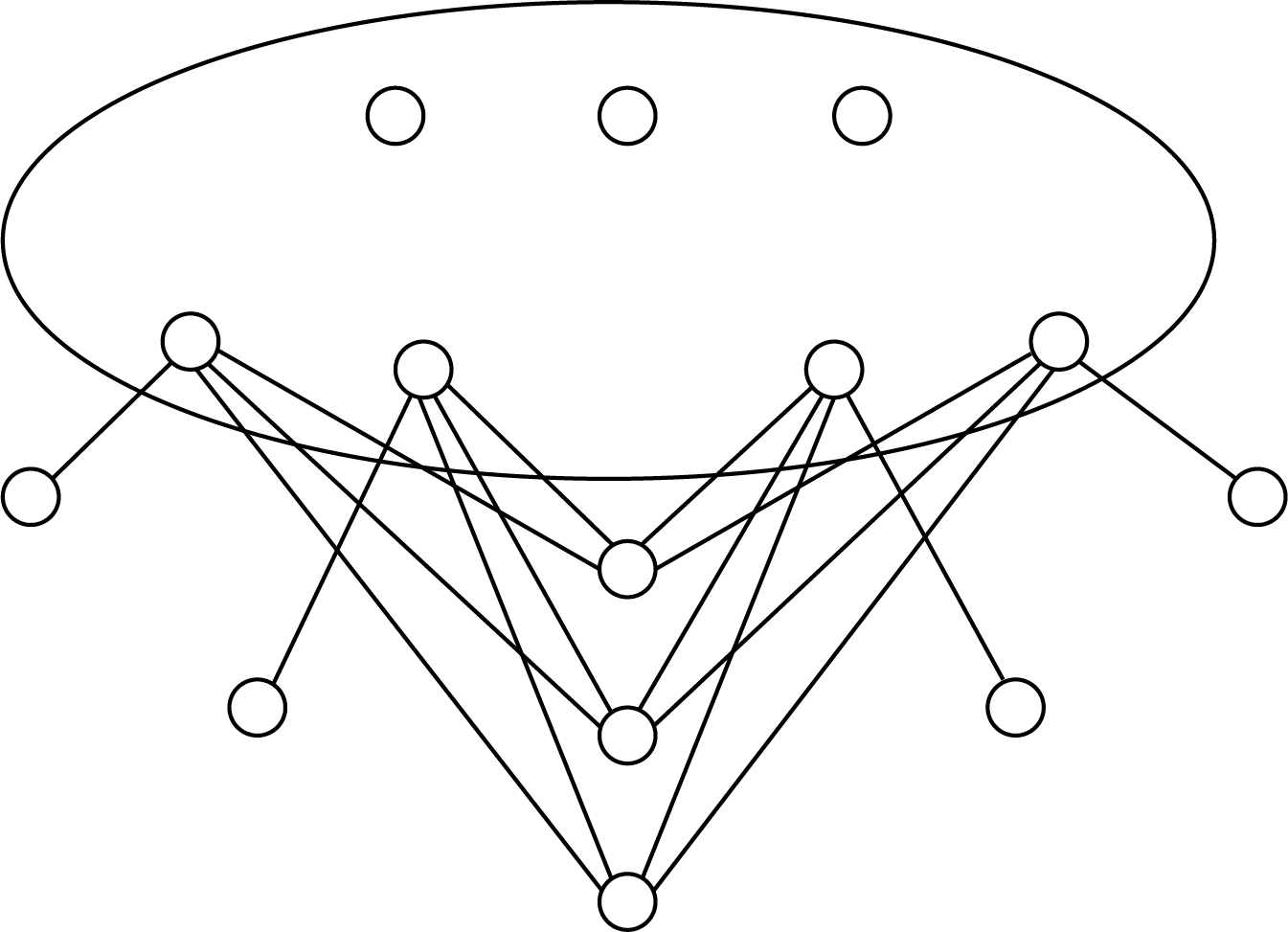}
\end{framed}
\end{figure}

\begin{proposition}
  Suppose that $G$ is a graph and $\edn(G)=1$. Then ${\adn(G)=1}$.
\end{proposition}
\begin{proof}
  If $\edn(G)=1$ then there is no independent set of vertices of size $2$. Therefore $G$ is necessarily a complete graph.
\end{proof}

The previous proposition is the reason for excluding the case $a=b=1$ and $c \ne 1$ in the following theorem. Otherwise there is no limit on how the various domination numbers can be related, beyond their ordering.

\begin{theorem}
  \label{thm:Realizability}
  Let $a$, $b$, and $c$ be natural numbers such that $a \le b \le c$ and either $c=1$ or $b>1$. Then there is a graph $G$ such that $\gamma(G)=a$, $\edn(G)=b$, and $\adn(G)=c$.
\end{theorem}
\begin{proof}
  It is necessary only to put together the examples collected previously. We organize the cases in a sort of lexicographic order, thinking of the sizes of $a$, $b$, and $c$. There are some extraordinary situations for small parameter values, leading to a number of cases.

\textit{Case $c=1$:} Then $K_n$ is such a graph.

\textit{Case $a=1$ and $b=2$}
If $c > 2$, then $\Acal_{c-1}$ is such a graph.
If $c=2$, then $P_3$ is such a graph.

\textit{Case $a=1$ and $b>2$}

The graph $\Ccal_{b-1,b-1 + c}$ is such a graph.

\textit{Case $a=2$ and $b=2$}
If $c=2$ then $P_4$ is such a graph.
If $c>2$ then $\Bcal_{0,c-1}$ is such a graph.

\textit{Case $a=2$ and $b=3$}
If $c=3$, then $K_{2,3}$ is such a graph.
If $c>3$, then $\Dcal_{1,c-2}$ is such a graph.

\textit{Case $a=2$ and $b > 3$}
If $b=c$, then $K_{2,b}$ is such a graph.
If $b<c$, then $\Ecal_{c-b,b-3}$ is such a graph.

\textit{Case $a \ge 3$ and $b=a$}
The graph $\Bcal_{b-2, c-b-1}$ is such a graph.

\textit{Case $a \ge 3$ and $b=a+1$}
The graph $\Dcal_{a-1,c-a}$ is such a graph.

\textit{Case $a \ge 3$ and $b=a\ge 2$}
The graph $\Fcal_{a,b-a-1,c-b+1}$ is such a graph.

\end{proof}

\section{Acknowledgements}

I thank Professor Kiran Bhutani for helpful remarks, which include her suggestion of the realizability question which is answered in Theorem~\ref{thm:Realizability}.

\bibliographystyle{plain}

\end{document}